\def\draftdate{\today}

\documentclass{amsart}
\usepackage{stmaryrd}
\usepackage{xy}
\usepackage{mathrsfs}
\usepackage{amscd,amssymb,color}
\usepackage[colorlinks=true]{hyperref}

\numberwithin{equation}{section}

\newtheorem{theorem}[equation]{Theorem}
\newtheorem*{theorem*}{Theorem}

\newtheorem{lemma}[equation]{Lemma}
\newtheorem{proposition}[equation]{Proposition}
\theoremstyle{definition}
\newtheorem{definition}[equation]{Definition}
\theoremstyle{remark}

\newtheorem{notation}[equation]{Notation}
\newtheorem{convention}[equation]{Convention}

\DeclareMathOperator{\Mot}{Mot}

\newcommand{\too}{\longrightarrow}
\newcommand{\dg}{\mathsf{dg}}
\newcommand{\dgHo}{\mathsf{H}^0}

\newcommand{\cA}{{\mathcal A}}
\newcommand{\cB}{{\mathcal B}}
\newcommand{\cC}{{\mathcal C}}
\newcommand{\cD}{{\mathcal D}}

\newcommand{\cJ}{{\mathcal J}}

\newcommand{\cM}{{\mathcal M}}

\newcommand{\cR}{{\mathcal R}}
\newcommand{\cS}{{\mathcal S}}
\newcommand{\cT}{{\mathcal T}}
\newcommand{\cU}{{\mathcal U}}

\newcommand{\cW}{{\mathcal W}}

\newcommand{\Mat}{{\mathsf M}_{\infty}} 

\newcommand{\bbD}{\mathbb{D}}

\newcommand{\bbL}{\mathbb{L}}

\newcommand{\bbK}{I\mspace{-6.mu}K}
\newcommand{\bbR}{\mathbb{R}}

\newcommand{\bbN}{\mathbb{N}}
\newcommand{\bbZ}{\mathbb{Z}}

\newcommand{\op}{\mathsf{op}} 
\newcommand{\ie}{\textsl{i.e.}\ }

\newcommand{\Hmo}{\mathsf{Hmo}}

\newcommand{\perf}{\mathsf{perf}} 

\newcommand{\Hom}{\mathsf{Hom}} 
\newcommand{\rep}{\mathsf{rep}} 
\newcommand{\dgcat}{\mathsf{dgcat}}
\newcommand{\HO}{\mathsf{HO}} 
\newcommand{\Madd}{\Mot^{\mathsf{add}}_{\dg}}

\newcommand{\Uadd}{\cU^{\mathsf{add}}_{\dg}}

\newcommand{\Mloc}{\Mot^{\mathsf{loc}}_{\dg}}
\newcommand{\Uloc}{\cU^{\mathsf{loc}}_{\dg}}

\newcommand{\uHom}{\underline{\mathsf{Hom}}}
\newcommand{\HomC}{\uHom_{!}}
\newcommand{\HomL}{\uHom_{\mathsf{loc}}}

\newcommand{\internalcomment}[1]{}
\xyoption{arrow}
\xyoption{matrix}
\xyoption{cmtip}
\SelectTips{cm}{}

\newdir{ >}{{}*!/-5pt/\dir{>}}

\begin{document}

\title[Universal suspension via non-commutative motives]{Universal suspension \\via non-commutative motives}
\author{Gon{\c c}alo~Tabuada}
\address{Departamento de Matem{\'a}tica e CMA, FCT-UNL, Quinta da Torre, 2829-516 Caparica,~Portugal}
\email{tabuada@fct.unl.pt}
\thanks{The author was partially supported by the Estimulo {\`a} Investiga{\c c}{\~a}o Award 2008 - Calouste Gulbenkian Foundation.}

\date{\draftdate}


\begin{abstract}
In this article we further the study of non-commutative motives, initiated in \cite{CT,CT1,Duke}. Our main result is the construction of a simple model, given in terms of infinite matrices, for the suspension in the triangulated category of non-commutative motives. As a consequence, this simple model holds in all the classical invariants such as Hochschild homology, cyclic homology and its variants (periodic, negative, $\ldots$), algebraic $K$-theory, topological Hochschild homology, topological cyclic homology, \ldots
\end{abstract}
\maketitle

\section{Introduction}
\subsection*{Non-commutative motives}
A {\em differential graded (=dg) category}, over a commutative base ring $k$, is a category enriched over 
complexes of $k$-modules (morphisms sets are such complexes)
in such a way that composition fulfills the Leibniz rule\,:
$d(f\circ~g)=(df)\circ g+(-1)^{\textrm{deg}(f)}f\circ(dg)$.
Dg categories enhance and solve many of the technical problems inherent to triangulated categories;
see Keller's ICM adress~\cite{ICM}. In {\em non-commutative algebraic geometry} in the sense of
Bondal, Drinfeld, Kaledin, Kapranov, Kontsevich, To{\"e}n,
Van den Bergh, $\ldots$ \cite{BKap,Bvan,Drinfeld,Kaledin,IAS,Kontsevich-Langlands,Toen},
they are considered as dg-enhancements of derived categories of (quasi-)coherent sheaves on a
hypothetic non-commutative space.

All the classical (functorial) invariants, such as Hochschild homology $HH$, cyclic homology $HC$, (non-connective) algebraic $K$-theory $\bbK$, topological Hochschild homology $THH$, and
topological cyclic homology $TC$, extend naturally from $k$-algebras to dg categories.
In order to study {\em all} these classical invariants simultaneously the author
introduced in \cite{Duke} the notion of {\em localizing invariant}. This notion, that we now recall,
makes use of the language of Grothendieck derivators~\cite{Grothendieck}, a formalism which
allows us to state and prove precise universal properties.
Let $\mathit{L}: \HO(\dgcat) \to \bbD$ be a morphism of derivators, from the
derivator associated to the derived Morita model structure on dg categories (see \S\ref{sub:Morita}), to
a triangulated derivator.
We say that $\mathit{L}$ is a {\em localizing invariant} if it preserves filtered homotopy
colimits as well as the terminal object, and sends exact sequences of dg categories~(see \S\ref{sub:exactseq})
\begin{eqnarray*}
\cA \too \cB \too \cC & \mapsto & \mathit{L}(\cA) \too \mathit{L}(\cB) \too \mathit{L}(\cC) \too \mathit{L}(\cA)[1]
\end{eqnarray*}
to distinguished triangles in the base category $\bbD(e)$ of $\bbD$.
Thanks to the work of Keller~\cite{Exact,Exact1}, Thomason-Trobaugh~\cite{Thomason}, Schlichting~\cite{Marco}, and Blumberg-Mandell~\cite{BM} (see also \cite{AGT}),
all the mentioned invariants satisfy localization\footnote{In the case of algebraic
$K$-theory we consider its non-connective version.}, and so give rise to localizing invariants. 
In \cite{Duke}, the author constructed the universal localizing invariant
$$ \Uloc: \HO(\dgcat) \too \Mloc\,,$$
\ie given any triangulated derivator $\bbD$, we have an induced equivalence of categories
\begin{equation}\label{eq:cat}
(\Uloc)^{\ast}: \HomC(\Mloc, \bbD) \stackrel{\sim}{\too} \HomL(\HO(\dgcat), \bbD)\,,
\end{equation}
where the left-hand side denotes the category of homotopy colimit preserving morphisms of derivators,
and the right-hand side denotes the category of localizing invariants.
Because of this universality property, which is a reminiscence of motives, $\Mloc$ is called the {\em localizing motivator}, and its base category $\Mloc(e)$ the {\em category of non-commutative motives}. We invite the reader to consult \cite{CT,CT1,Duke} for several applications of this theory of non-commutative motives.
\subsection*{Universal suspension}
The purpose of this article is to construct a simple model for the suspension in the triangulated category of non-commutative motives. 

Consider the $k$-algebra $\Gamma$ of $\bbN\times \bbN$-matrices $A$ which satisfy the following two conditions\,: (1) the set $\{A_{i,j}\, |\, i, j \in \bbN \}$ is finite; (2) there exists a natural number $n_A$ such that each row and each column has at most $n_A$ non-zero entries; see Definition~\ref{def:cone}. Let $\Sigma$ be the quotient of $\Gamma$ by the two-sided ideal consisting of those matrices with finitely many non-zero entries; see Definition~\ref{def:finMat}. Alternatively, take the (left) localization of $\Gamma$ with respect to the matrices $\overline{I_n}, n \geq 0$, with entries $(\overline{I_n})_{i,j}={\bf 1}$ for $i=j >n$ and $0$ otherwise; see Proposition~\ref{prop:alt-descr}. The algebra $\Sigma$ goes back to the work of Karoubi and Villamayor~\cite{Karoubi} on negative $K$-theory. Recently, it was used by Corti{\~n}as and Thom~\cite{Cortinas} in the construction of a bivariant algebraic $K$-theory.
Given a dg category $\cA$, we denote by $\Sigma(\cA)$ the tensor product of $\cA$ with $\Sigma$; see \S\ref{sub:bimodules}. The main result of this article is the following.
\begin{theorem}\label{thm:main}
For every dg category $\cA$ we have a canonical isomorphism
$$ \Uloc(\Sigma(\cA)) \stackrel{\sim}{\too} \Uloc(\cA)[1]\,.$$
\end{theorem}
The proof of Theorem~\ref{thm:main} is based on several properties of the category of non-commutative motives (see Section~\ref{sub:proof}), on an exact sequence relating $\cA$ and $\Sigma(\cA)$ (see Section~\ref{sub:exact}), and on the {\em flasqueness} of $\Gamma$ (see Section~\ref{sub:flasque}). Let us now describe some applications of Theorem~\ref{thm:main}.
\subsection*{Applications}
A {\em realization} of the category of non-commutative motives is a triangulated functor $R: \Mloc(e) \to \cT$. An important aspect of a realization is the fact that {\em every} result which holds on $\Mloc(e)$ also holds on $\cT$. In particular, given a dg category $\cA$, Theorem~\ref{thm:main} furnish us a canonical isomorphism
\begin{equation*}
(R\circ \Uloc)(\Sigma(\cA)) \stackrel{\sim}{\too} (R\circ \Uloc)(\cA)[1]\,.
\end{equation*}
Thanks to the above equivalence \eqref{eq:cat} every localizing invariant gives rise to a realization. Therefore, we obtain the canonical isomorphisms\,: 
\begin{eqnarray}
HH(\Sigma(\cA)) \simeq HH(\cA)[1] && HH_{\ast+1}(\Sigma(\cA))\simeq HH_{\ast}(\cA) \label{eq:isom-1}\\
HC(\Sigma(\cA)) \simeq HC(\cA)[1] && HC_{\ast+1}(\Sigma(\cA))\simeq HC_{\ast}(\cA) \label{eq:isom-2}\\
\bbK(\Sigma(\cA)) \simeq \bbK(\cA)[1] && \bbK_{\ast+1}(\Sigma(\cA))\simeq \bbK_{\ast}(\cA) \label{eq:isom-3}\\
THH(\Sigma(\cA)) \simeq THH(\cA)[1] && THH_{\ast+1}(\Sigma(\cA))\simeq THH_{\ast}(\cA) \label{eq:isom-4}\\
TC(\Sigma(\cA)) \simeq TC(\cA)[1] && TC_{\ast+1}(\Sigma(\cA))\simeq TC_{\ast}(\cA) \label{eq:isom-5}\,. 
\end{eqnarray}
Negative cyclic homology $HC^-$ and periodic cyclic homology $HP$ are not examples of localizing invariants since they do {\em not} preserve filtered (homotopy) colimits. Nevertheless, as explained in \cite[Examples~8.10 and 8.11]{CT1}, they factor through $\Mloc$ thus giving rise to realizations. We obtain then the canonical isomorphisms\,:
\begin{eqnarray}
HC^-(\Sigma(\cA)) \simeq HC^-(\cA)[1] && HC^-_{\ast+1}(\Sigma(\cA))\simeq HC^-_{\ast}(\cA) \label{eq:isom-6} \\
HP(\Sigma(\cA))\simeq HP(\cA)[1] && HP_{\ast+1}(\Sigma(\cA))\simeq HP_{\ast}(\cA) \label{eq:isom-7}\,.
\end{eqnarray}
Note that since $HP$ is $2$-periodic, the homologies of $\Sigma(\cA)$ and $\cA$ can be obtained from each other by simply switching the degrees. To the best of the author's knowledge the isomorphisms \eqref{eq:isom-1}-\eqref{eq:isom-7} are new. They show us that $\Sigma(\cA)$ is a simple model for the suspension in all these classical invariants.\footnote{Recall that all these invariants take values in arbitrary degrees.}

We would like to mention that Kassel constructed an isomorphism related to \eqref{eq:isom-2} but for ordinary algebras over a field and with cyclic homology replaced by bivariant cyclic cohomology; see~\cite[Theorem~3.1]{Kassel}. Instead of $\Gamma$, he considered the larger algebra of infinite matrices which have finitely many non-zero entries in each line and column.

Now, let $X$ a quasi-compact and quasi-separated scheme. It is well-known that the category of perfect complexes in the (unbounded) derived category of quasi-coherent sheaves on $X$ admits a dg-enhancement $\perf_\dg(X)$; see for instance~\cite{BKap,Orlov} or~\cite[Example~4.5]{CT1}. Thanks to \cite[Theorem~1.3]{BM}, \cite[\S5.2]{Exact1} and \cite[\S8 Theorem~5]{Marco} the algebraic $K$-theory and the (topological) cyclic homology\footnote{In fact we can consider any variant of (topological) cyclic homology.} of the scheme $X$ can be obtained from the dg category $\perf_\dg(X)$ by applying the corresponding invariant. Therefore, when $\cA=\perf_\dg(X)$, the above isomorphisms \eqref{eq:isom-1}-\eqref{eq:isom-7} suggest us that the dg category $\Sigma(\perf_\dg(X))$ should be considered as the ``non-commutative suspension'' or ``non-commutative delooping'' of the scheme $X$. This will be the subject of future research.

\subsection*{Acknowledgments\,:} Theorem~\ref{thm:main} answers affirmatively a question raised by Maxim Kontsevich in my Ph.D. thesis defense~\cite{Thesis}. I deeply thank him for his insight. I am also grateful to Bernhard Keller, Marco Schlichting and Bertrand To{\"e}n for useful conversations and/or references.

\begin{convention}
Throughout the article $k$ will denote a commutative base ring with unit ${\bf 1}$.
Given a dg algebra $H$ we will denote by $\underline{H}$ the dg category with a single object $\ast$ and with $H$ as the dg algebra of endomorphisms. 
\end{convention}
\section{Background on dg categories}\label{sec:background}
In this section we collect some notions and results on dg categories which will be used throughout the article.

Let $\cC(k)$ be the category of (unbounded) complexes of $k$-modules; we use cohomological notation. A {\em differential graded (=dg) category} is a category enriched over $\cC(k)$ and a {\em dg functor} is a functor enriched over $\cC(k)$; consult Keller's ICM adress~\cite{ICM} for a survey on dg categories. The category of dg categories will be denoted by $\dgcat$.
\begin{notation}\label{not:basic}
Let $\cA$ be a dg category. The category $\mathsf{Z}^0(\cA)$ has the same objects as $\cA$ and morphisms given by $\mathsf{Z}^0(\cA)(x,y):=\textrm{Z}^0(\cA(x,y))$. The category $\dgHo(\cA)$ has the same objects as $\cA$ and morphisms given by $\dgHo(\cA)(x,y):=\textrm{H}^0(\cA(x,y))$. The {\em opposite} dg category $\mathcal{A}^{\op}$ of $\cA$ has the same objects as $\mathcal{A}$ and complexes of morphisms given by
$\mathcal{A}^{\op}(x,y):=\mathcal{A}(y,x)$.
\end{notation}
\subsection{(Bi)modules}\label{sub:bimodules}
Let $\cA$ be a dg category. A {\em right $\cA$-module} $M$ is a dg functor
$M:\cA^{\op} \to \cC_{\dg}(k)$ with values in the dg category $\cC_{\dg}(k)$ of complexes of
$k$-modules. We will denote by $\cC(\cA)$ the category of right $\cA$-modules; see \cite[\S 2.3]{ICM}. As explained in \cite[\S 3.1]{ICM} the differential graded structure of $\cC_{\dg}(k)$ makes $\cC(\cA)$
naturally into a dg category $\cC_\dg(\cA)$. Recall from~\cite[Theorem~3.2]{ICM} that $\cC(\cA)$ carries a standard projective $\cC(k)$-model structure. The {\em derived category $\cD(\cA)$ of $\cA$} is the localization of
$\cC(\cA)$ with respect to the class of objectwise quasi-isomorphisms.
\begin{notation}\label{not:perf}
We denote by $\perf(\cA)$, resp. by $\perf_\dg(\cA)$, the full subcategory of $\cC(\cA)$, resp. full dg subcategory of $\cC_\dg(\cA)$, whose objects are the cofibrant right $\cA$-modules that are compact~\cite[Definition~4.2.7]{Neeman} in the triangulated category $\cD(\cA)$.
\end{notation}
Given dg categories $\cA$ and $\cB$ their {\em tensor product}
$\cA \otimes \cB$ is defined as follows: the set of objects is the cartesian product and given objects $(x,z)$ and $(y,w)$ in $\cA \otimes \cB$, we set $(\cA \otimes \cB)((x,z),(y,w)):= \cA(x,y) \otimes \cB(z,w)$. A {\em $\cA\text{-}\cB$-bimodule} $X$ is a dg functor $X:\cA^\op\otimes \cB \to \cC_\dg(k)$, \ie a right $\cA^\op\otimes \cB$-module.
\subsection{Derived Morita equivalences}\label{sub:Morita}
A dg functor $F: \cA \to \cB$ is a called a {\em derived Morita equivalence}
if its derived extension of scalars functor $\bbL F_!: \cD(\cA) \stackrel{\sim}{\to}~\cD(\cB)$ (see \cite[\S3]{Toen}) is an equivalence of triangulated categories. Thanks to \cite[Theorem~5.3]{IMRN} (and \cite{IMRNC}) the category $\dgcat$ carries a (cofibrantly generated) Quillen model structure whose weak equivalences are the derived Morita equivalences. We denote by $\Hmo$ the homotopy category hence obtained.

The tensor product of dg categories can be derived into a bifunctor $-\otimes^\bbL-$ on $\Hmo$. Moreover, thanks to \cite[Theorem.~6.1]{Toen} the bifunctor $-\otimes^\bbL-$ admits an internal Hom-functor $\rep(-,-)$.\footnote{Denoted by $\bbR\uHom(-,-)$ in {\em loc. cit.}} Given dg categories $\cA$ and $\cB$, $\rep(\cA,\cB)$ is the full dg subcategory of $\cC_\dg(\cA^\op\otimes^{\bbL}\cB)$ spanned by the cofibrant $\cA\text{-}\cB$-bimodules $X$ such that, for every object $x$ in $\cA$, the right $\cB$-module $X(x,-)$ is compact in $\cD(\cB)$. The set of morphisms in $\Hmo$ from $\cA$ to $\cB$ is given by the set of isomorphism classes of the triangulated category $\dgHo(\rep(\cA,\cB))$.
\subsection{Exact sequences}\label{sub:exactseq}
A sequence of triangulated categories
$$0 \too \cR \stackrel{I}{\too} \cS \stackrel{P}{\too} \cT \too 0$$
is called {\em exact} if the composition is zero, the functor $I$ is
fully-faithful and the induced functor from the Verdier quotient $\cS/\cR$ to $\cT$
is {\em cofinal}, \ie it is fully-faithful and every object in $\cT$ is a direct summand
of an object of $\cS/\cR$; see \cite[\S2]{Neeman}. A sequence in $\Hmo$
$$0 \too \cA \stackrel{X}{\too} \cB \stackrel{Y}{\too} \cC \too 0$$
is called {\em exact} if the induced sequence of triangulated categories
$$0 \too \cD(\cA) \stackrel{-\otimes_\cA^\bbL X}{\too} \cD(\cB) \stackrel{-\otimes_\cB^\bbL Y}{\too} \cD(\cC) \too 0$$
is exact; see \cite[\S4.6]{ICM}.
\section{Infinite matrix algebras}\label{sec:matrices}
In this section we introduce the matrix algebras used in the construction of the universal suspension.
\begin{definition}\label{def:finMat}
Given $n \in \bbN$, we denote by $\mathsf{M}_n$ the $k$-algebra of $n\times n$-matrices with coefficients in $k$. Let
$$\Mat:= \bigcup_{n=1}^\infty\mathsf{M}_n$$
be the $k$-algebra of {\em finite matrices}, where $\mathsf{M}_n \subset \mathsf{M}_{n+1}$ via the map
$$ A \mapsto \left[\begin{matrix} A  & 0  \\ 0 & 0 \end{matrix}\right]\,.$$
Note that $\Mat$ does {\it not} have a unit object. Moreover, transposition of matrices gives rise to an isomorphism of $k$-algebras
\begin{equation}\label{eq:transp1}
(-)^T: (\Mat)^\op \stackrel{\sim}{\too} \Mat\,.
\end{equation}
\end{definition} 
\begin{notation}\label{not:elem-matrices}
Given $k, l \in \bbN$, we denote by $E_{kl} \in \Mat$ the matrix 
\begin{equation*}
(E_{kl})_{i,j} := \left\{ \begin{array}{lcl}
{\bf 1} &  \text{if} & i=k \,\, \text{and}\,\,j=l \\
0 &  & \text{otherwise} \\
\end{array} \right.
\end{equation*}
Note that given $k, l, m, n \in \bbN$, the product $E_{kl} \cdot E_{nm}$ equals $E_{km}$ if $l=n$ and is zero otherwise. Given a non-negative integer $n \geq 0$, we denote by $I_n \in \Mat$ the matrix 
\begin{equation*}
(I_{n})_{i,j} := \left\{ \begin{array}{lcl}
{\bf 1} &  \text{if} & i=j\leq n \\
0 &  & \text{otherwise} \\
\end{array} \right.
\end{equation*}
In particular, $I_0$ stands for the zero matrix.
\end{notation}
\begin{lemma}\label{lem:localunits}
The $k$-algebra $\Mat$ has {\em idempotent local units}, \ie for each finite family $A_s, s \in S$, of elements in $\Mat$ there exists an idempotent $E\in \Mat$ such that $E \cdot A_s = A_s \cdot E =A_s$ for all $s \in S$.
\end{lemma}
\begin{proof}
Since the matrices $A_s, s \in S$, have only a finite number of non-zero entries there exist natural numbers $m_s, s \in S$, such that $(A_s)_{i,j}=0$ when $i$ or $j$ is greater than $m_s$. Let $m:=\mathrm{max}\{m_s\,|\,s \in S\}$. Then, if $E$ is the idempotent matrix $I_m$ we observe that $I_m \cdot A_s = A_s \cdot I_m=A_s$ for all $s \in S$.
\end{proof}

\begin{definition}\label{def:cone}
Let $\Gamma$ be the $k$-algebra of $\bbN\times \bbN$-matrices $A$ with coefficients in $k$ and satisfying the following two conditions\,:
\begin{itemize}
\item[(1)] the set $\{A_{i,j}\, |\, i, j \in \bbN \}$ is finite;
\item[(2)] there exists a natural number $n_A$ (which depends on $A$) such that each row and each column has at most $n_A$ non-zero entries. 
\end{itemize}
The $k$-module structure is defined entrywise and the multiplication is given by the ordinary matrix multiplication law; note that if $A, B \in \Gamma$ we can take $n_A \times n_B$ as the natural number $n_{A\cdot B}$. In contrast with $\Mat$, $\Gamma$ does {\em have} a unit object
\begin{equation*}
I_{i,j} := \left\{ \begin{array}{lcl}
{\bf 1} &  \text{if} & i=j\\
0 &  & \text{otherwise} \\
\end{array} \right.
\end{equation*}
Moreover, transposition of matrices induces an isomorphism of $k$-algebras
\begin{equation*}
(-)^T: \Gamma^\op \stackrel{\sim}{\too} \Gamma
\end{equation*}
which extends isomorphism \eqref{eq:transp1}.
\end{definition}
Now, let us fix a bijection
\begin{eqnarray*}
\theta: \bbN \stackrel{\sim}{\too} \bbN \times \bbN && n \mapsto (\theta_1(n), \theta_2(n))\,;
\end{eqnarray*}
take for instance the inverse of Cantor's classical pairing function. 
As in \cite[Lemma~19]{Schlichting}, we define a $k$-algebra homomorphism
\begin{eqnarray*}
\phi: \Gamma \too \Gamma && A \mapsto \phi(A)
\end{eqnarray*}
as follows 
\begin{equation*}
\phi(A)_{i,j} := \left\{ \begin{array}{lcl}
A_{\theta_1(i),\theta_1(j)} &  \text{if}  & \theta_2(i)=\theta_2(j) \\
0 &  & \text{otherwise} \\
\end{array} \right.
\end{equation*}
Note that the non-zero elements in line $i$, resp. in column $j$, of the matrix $\phi(A)$ are precisely the non-zero elements in line $i$, resp. in column $j$, of the matrix $A$. 
\begin{definition}\label{def:bimodule}
Let $W$ be the $\Gamma\text{-}\Gamma$-bimodule, which is $\Gamma$ as a left $\Gamma$-module, and whose right $\Gamma$-action is given by
\begin{eqnarray*}
\Gamma \times \Gamma \too \Gamma && (B,A) \mapsto B \cdot \phi(A)\,.
\end{eqnarray*}   
\end{definition}
\begin{lemma}\label{lem:flasque}
There exists a natural $\Gamma\text{-}\Gamma$-bimodule isomorphism $\Gamma\oplus W \stackrel{\sim}{\to} W$.
\end{lemma}
\begin{proof}
Consider the elements
\begin{eqnarray*}
\alpha_{i,j} := \left\{ \begin{array}{lcl}
{\bf 1} &  \text{if} & \theta(j)=(i,0) \\
0 &  & \text{otherwise} \\
\end{array} \right.  \
\end{eqnarray*}
and
\begin{eqnarray*}
\beta_{i,j} := \left\{ \begin{array}{lcl}
{\bf 1} &  \text{if} & \theta(j)=\theta(i)+(0,1) \\
0 &  & \text{otherwise} \\
\end{array} \right. 
\end{eqnarray*}
in $\Gamma$. Using $\alpha$ and $\beta$, we define maps\,:
\begin{eqnarray}
\Gamma\oplus W \too W && (A,B) \mapsto A \cdot \alpha + B \cdot \beta \label{eq:map1}\\
W \too \Gamma\oplus W && B \mapsto (B \cdot \alpha^T, B \cdot \beta^T)\label{eq:map2} \,. 
\end{eqnarray}
The map \eqref{eq:map1} is a left $\Gamma$-module homomorphism. The fact that it is also a right $\Gamma$-module homomorphism follows from the following equalities\,:
\begin{eqnarray*}
\beta \cdot \alpha^T =0 & \alpha \cdot \alpha^T = \beta \cdot \beta^T = I & \alpha^T \cdot \alpha + \beta^T \cdot \beta = I \,.
\end{eqnarray*}
Moreover, since for every $A \in \Gamma$ we have 
\begin{eqnarray*}
A\cdot \alpha = \alpha \cdot \phi(A) & \mathrm{and} & \phi(A)\cdot \beta = \beta \cdot \phi(A)
\end{eqnarray*}
we conclude that the maps \eqref{eq:map1} and \eqref{eq:map2} are inverse of each other.
\end{proof}
\begin{notation}\label{not:suspension}
Clearly the $k$-algebra $\Mat$ forms a two-sided ideal in $\Gamma$. We denote by $\Sigma$ the associated quotient $k$-algebra $\Gamma/\Mat$.
\end{notation}
Alternatively, we can describe the quotient $k$-algebra $\Sigma$ as follows.
\begin{proposition}\label{prop:alt-descr}
The matrices 
\begin{eqnarray*}
\overline{I_n}:= I - I_n && (\mathrm{see\,\,Notation}~\ref{not:elem-matrices})
\end{eqnarray*}
form a left denominator set S in $\Gamma$~\cite[\S4]{Lam}, \ie $I \in S$, $S\cdot S \subset S$ and
\begin{itemize}
\item[(i)] given $\overline{I_n} \in S$ and $E \in \Gamma$, there are $\overline{I_m} \in S$ and $E' \in \Gamma$ such that $E'\cdot\overline{I_n} =\overline{I_m} \cdot E$;
\item[(ii)] if $\overline{I_n} \in S$ and $E\in \Gamma$ satisfy $E \cdot \overline{I_n}=0$, there is $\overline{I_m}\in S$ such that $\overline{I_m}\cdot E=0$.
\end{itemize}
Moreover, the localized $k$-algebra $\Gamma[S^{-1}]$\footnote{Since $S$ is a left denominator set this $k$-algebra is given by left fractions, \ie equivalence classes of pairs $(\overline{I_n},E)$ modulo the relation which identifies $(\overline{I_n},E)$ with $(\overline{I_m},E')$ if there are $B, B' \in \Gamma$ such that $B\cdot \overline{I_n} = B'\cdot \overline{I_m}$ belongs to $S$ and $B\cdot E = B'\cdot E'$.} is naturally isomorphic to $\Sigma$.
\end{proposition}
\begin{proof}
In order to simplify the proof  we consider the following block-matrix graphical notation
\begin{equation*}
E = \text{\scriptsize $k$} \overset{l}{\left[
\begin{array}{c|c}
 E_a& E_b\\ \hline
E_c & E_d
\end{array}\right]
} \in \Gamma\,,
\end{equation*}
where $k, l \in \bbN$, $E_a$ is a $k \times l$-matrix, $E_b$ is a $k \times \bbN$-matrix, $E_c$ is a $\bbN\times l$-matrix, and $E_d$ is a $\bbN \times \bbN$-matrix. Under this notation we have, for $n \in \bbN$, the equalities
\begin{equation}\label{eq:graphic1}
\overline{I_n} \cdot \text{\scriptsize $n$}\overset{n}{\left[
\begin{array}{c|c}
 E_a& E_b\\ \hline
 E_c& E_d
\end{array}\right]
}  =  \text{\scriptsize $n$}\overset{n}{\left[
\begin{array}{c|c}
 0&0 \\ \hline
 E_c& E_d
\end{array}\right]
}   
\end{equation}
and
\begin{equation}\label{eq:graphic2}
\text{\scriptsize $n$}\overset{n}{\left[
\begin{array}{c|c}
 E_a&E_b \\ \hline
 E_c& E_d
\end{array}\right]
} \cdot \overline{I_n} = \text{\scriptsize $n$}\overset{n}{\left[
\begin{array}{c|c}
 \,0\,\,&E_b \\ \hline
 \,0\,\,& E_d
\end{array}\right]
} \,.
\end{equation}
By definition $I = \overline{I_0} \in S$. Equalities \eqref{eq:graphic1} and \eqref{eq:graphic2}, and the fact that $\overline{I_0}=I$, imply that
\begin{eqnarray*} 
\overline{I_n}\cdot \overline{I_m} = \overline{I_{\mathrm{max}\{n,m\}}} && n \geq 0\,.
\end{eqnarray*}
This shows that $S\cdot S \subset S$.

(i) Note first that when $n=0$ the claim is trivial. Since $E$ belongs to $\Gamma$, there exist natural numbers $m_j, 1\leq j\leq n$, such that $E_{i,j}=0$ for $ i \geq m_j$ and $1\leq j \leq n$. Take $m =\mathrm{max}\{n, m_j\,|\, 1 \leq j \leq n \}$. Then, we have the following equality
\begin{equation}\label{eq:matrix} 
\overline{I_m} \cdot  \text{\scriptsize $m$}\overset{n}{\left[
\begin{array}{c|c}
 E_a&E_b \\ \hline
 E_c& E_d
\end{array}\right]
} = \text{\scriptsize $m$} \overset{n}{\left[
\begin{array}{c|c}
 0\,\,\, &0  \\ \hline
 0\,\,\,&E_d 
\end{array}\right]
} \,.
\end{equation}
Since $m \geq n$, the above equality \eqref{eq:graphic2} shows us that we can take for $E'$ the above matrix \eqref{eq:matrix}. This proves the claim.

(ii) When $n=0$ the claim is trivial. If $E\cdot \overline{I_n}=0$ the above equality \eqref{eq:graphic2} shows us that 
\begin{equation}\label{eq:descr}
E =\text{\scriptsize $n$}\overset{n}{\left[
\begin{array}{c|c}
 E_a \!&\,\,0\, \\ \hline
 E_c\!&\, \,0 \, 
\end{array}\right]
}.
\end{equation}
Since $E$ belongs to $\Gamma$, there exist natural numbers $m_j, 1 \leq j \leq n$, such that $E_{i,j}=0$ for $i \geq m_j$ and $1 \leq j \leq n$. Take $m =\mathrm{max}\{m_j\,|\, 1 \leq j \leq n\}$. Then, the above description \eqref{eq:descr} combined with equality \eqref{eq:graphic2} show us that $\overline{I_m}\cdot E=0$. This proves the claim.

We now show that the localized $k$-algebra $\Gamma[S^{-1}]$ is naturally isomorphic to $\Sigma$. Since the matrices 
\begin{eqnarray*}
I_n=I-\overline{I_n}&&n \geq 0
\end{eqnarray*}
belong to $\Mat$, we conclude that all the elements of the set $S$ become the identity object in $\Sigma$. Therefore, by the universal property of $\Gamma[S^{-1}]$ we obtain a $k$-algebra map
\begin{equation}\label{eq:k-alg1}
\Gamma[S^{-1}] \too \Sigma\,.
\end{equation}
On the other hand, the kernel of the localization map $\Gamma \to \Gamma[S^{-1}]$ consists of those matrices $E \in \Gamma$ for which $\overline{I_n}\cdot E=0$ for some $n \geq 0$. Thanks to equality \eqref{eq:graphic1} we observe that the elements of $\Mat$ satisfy this condition. Therefore, by the universal property of $\Sigma:= \Gamma/\Mat$ we obtain a $k$-algebra map
\begin{equation}\label{eq:k-alg2}
\Sigma \too \Gamma[S^{-1}]\,.
\end{equation}
The maps \eqref{eq:k-alg1} and \eqref{eq:k-alg2} are clearly inverse of each other and so the proof is finished.
\end{proof}

\begin{lemma}\label{lem:flat}
The algebras $\Mat$, $\Gamma$ and $\Sigma$ are flat as $k$-modules.
\end{lemma}
\begin{proof}
We start by proving this proposition in the particular case where the base ring $k$ is $\bbZ$. In this case the underlying $\bbZ$-modules of $(\Mat)_{\bbZ}$ and $\Gamma_{\bbZ}$ are torsionfree and so by \cite[Corollary~3.1.5]{Weibel}, they are flat. Thanks to Proposition~\ref{prop:alt-descr}, $\Sigma_\bbZ$ identifies with the (left) localization of $\Gamma_\bbZ$ with respect to the set $S$, and so a standard argument shows us that the right $\Gamma_{\bbZ}$-module $\Sigma_{\bbZ}$ is flat. Since $\Gamma_{\bbZ}$ is flat as a $\bbZ$-module, we conclude that $\Sigma_{\bbZ}$ is also flat as a $\bbZ$-module.

Let us now consider the general case. Clearly we have a natural isomorphism of $k$-modules
$$ (\Mat)_{\bbZ} \otimes_{\bbZ} k \stackrel{\sim}{\too} (\Mat)_k\,. $$
Thanks to \cite[Lemma~4.7.1]{Cortinas}, we have also natural isomorphisms of $k$-modules
\begin{eqnarray*}
\Gamma_{\bbZ} \otimes_{\bbZ} k \stackrel{\sim}{\too} \Gamma_k & \mathrm{and} & \Sigma_{\bbZ} \otimes_{\bbZ} k \stackrel{\sim}{\too} \Sigma_k\,.
\end{eqnarray*}
Therefore, since flat modules are stable under extension of scalars, the proof is achieved.
\end{proof}

\section{An exact sequence}\label{sub:exact}
Let $H$ be a $k$-algebra and $J \subset H$ a two-sided ideal.
\begin{definition}
The {\em category $\cJ$ of idempotents of $J$} is defined as follows\,: its objects are the symbols ${\bf u}$, where $u$ is an idempotent of $J$; the $k$-module $\cJ(\mathbf{u}, {\bf u'})$ of morphisms from ${\bf u}$ to ${\bf u'}$ is $uJu'$; composition is given by multiplication in $J$ and the unit of each object ${\bf u}$ is the idempotent $u$.
Associated to $H$ and $J$ there is also a $\cJ\text{-}\underline{H}$-bimodule $X$ such that $X({\bf u}, \ast):= uJ$, with left and right actions given by multiplication.
\end{definition}
Recall from \cite[Example~3.3(b)]{Keller} that if $H$ and $J$ are flat as $k$-modules and $J$ has {\em idempotent local units} (\ie for each finite family $a_s, s \in S$, of elements in $J$ there exists an idempotent $u \in J$ such that $u a_s =a_s u = a_s$ for all $s \in S$) we have a exact sequence in $\Hmo$
\begin{equation*}
0 \too \cJ \stackrel{X}{\too} \underline{H} \too \underline{H/J} \too 0\,.
\end{equation*}
Thanks to Lemmas~\ref{lem:localunits} and~\ref{lem:flat}, if we take $H=\Gamma$ and $J= \Mat$, we then obtain then the following exact sequence in $\Hmo$
\begin{equation}\label{eq:ses1}
0 \too \cM_{\infty} \stackrel{X}{\too} \underline{\Gamma} \too \underline{\Sigma} \too 0\,.
\end{equation}
\begin{proposition}
The dg functor
\begin{eqnarray}\label{eq:Moritaeq}
\underline{k} \too \cM_\infty &  \ast \mapsto {\bf E_{11}}& (\text{see\,\,Notation}~\ref{not:elem-matrices})
\end{eqnarray}
is a derived Morita equivalence.
\end{proposition}
\begin{proof}
We will prove a stronger statement, namely that the above functor \eqref{eq:Moritaeq} is a Morita equivalence; see \cite[\S2]{Schwede}. The category $\cM_\infty$ is by definition enriched over $k$-modules and the classical theory of Morita holds in this setting. Let $\mathrm{Mod}\text{-}\cM_\infty$ be the abelian category of right $\cM_\infty$-modules (\ie contravariant $k$-linear functors from $\cM_\infty$ to $k$-modules) and
\begin{eqnarray*}
\widehat{(-)}: \cM_\infty \too \mathrm{Mod}\text{-}\cM_\infty && {\bf E} \mapsto \cM_\infty(-, {\bf E})=: \widehat{{\bf E}}
\end{eqnarray*}
the (enriched) Yoneda functor. Following \cite[Theorems~2.2 and 2.5]{Schwede}, we need to show that $\widehat{{\bf E_{11}}}$ is a small projective generator of $\mathrm{Mod}\text{-}\cM_\infty$ and that its ring of endomorphisms is isomorphic to $k$. Note that we have natural isomorphisms
$$ \Hom_{\mathrm{Mod}\text{-}\cM_\infty}(\widehat{{\bf E_{11}}}, \widehat{{\bf E_{11}}}) \simeq \cM_\infty({\bf E_{11}}, {\bf E_{11}})= E_{11}\cdot \Mat \cdot E_{11} \simeq k\,.$$
Moreover, $\widehat{{\bf E_{11}}}$ is small and projective by definition. Therefore, it only remains to show that $\widehat{{\bf E_{11}}}$ is a generator, \ie, that every right $\cM_\infty$-module $P$ is an epimorphic image of a sum of (possibly infinitely many) copies of $\widehat{{\bf E_{11}}}$. Given an object ${\bf E}$ in $\cM_\infty$ we have, by the (enriched) Yoneda lemma, an isomorphism
$$ \Hom_{\mathrm{Mod}\text{-}\cM_\infty}(\widehat{{\bf E}}, P) \simeq P({\bf E})$$
and so we obtain a natural epimorphism
$$
\xymatrix{
\bigoplus_{{\bf E} \in \cM_\infty} \bigoplus_{P({\bf E})} \widehat{{\bf E}} \ar@{->>}[r] & P\,.
}$$
This shows that it suffices to treat the case where $P$ is of shape $\widehat{{\bf E}}$. We consider first the cases ${\bf E}={\bf E_{nn}}, n \in \bbN$. The following morphisms in $\cM_\infty$
$$
\xymatrix{
{\bf E_{11}} \ar[rr]^{E_{11}\cdot E_{1n}\cdot E_{nn}} && {\bf E_{nn}} & \mathrm{and} & {\bf E_{nn}} \ar[rr]^{E_{nn}\cdot E_{n1}\cdot E_{11}} && {\bf E_{11}} \\
}
$$
show us that ${\bf E_{11}}$ and ${\bf E_{nn}}$ are isomorphic and so the claim follows.

We consider now the cases ${\bf E}={\bf I_m}, m \in \bbN$. The natural morphisms in $\cM_\infty$
$$
\xymatrix{
{\bf E_{nn}} \ar[rr]^-{E_{nn} \cdot E_{nn} \cdot I_m} && {\bf I_m} &  1\leq n \leq m
}
$$
give rise to a map in $\mathrm{Mod}\text{-}\cM_\infty$
\begin{equation}\label{eq:induced}
\bigoplus_{n \geq 1}^m \, \widehat{{\bf E_{nn}}} \too \widehat{{\bf I_m}}\,.
\end{equation}
In order to show that the map \eqref{eq:induced} is surjective, we need to show that its evaluation
\begin{equation}\label{eq:induced1}
\bigoplus_{n \geq 1}^m \, \cM_\infty({\bf B}, {\bf E_{nn}}) \too \cM_\infty({\bf B}, {\bf I_m})
\end{equation}
at each object ${\bf B}$ of $\cM_\infty$ is surjective. We have $I_m = \sum_{n \geq 1}^m E_{nn}$, and so \eqref{eq:induced1} identifies with the natural map
$$ \bigoplus_{n\geq 1}^m\, \left(B\cdot \Mat \cdot E_{nn}\right) \too B \cdot \Mat \cdot \left(\sum_{n\geq 1}^m E_{nn}\right)\,,$$
which is easily seen to be surjective. Since ${\bf E_{11}}$ is isomorphic to ${\bf E_{nn}}, n \in \bbN$, the claim is proved.

Finally, we consider the case of a general object ${\bf E}$ in $\cM_\infty$. Since $E$ has only a finite number of non-zero entries there exists a natural number $m$ such that $E_{i,j}=0$ when $i$ or $j$ is greater than $m$. We have then the equality
$$E \cdot I_m = I_m \cdot E = E\,.$$
This implies that the composition
$$
\xymatrix{
{\bf E} \ar[rr]^{E \cdot I_m \cdot I_m} && {\bf I_m} \ar[rr]^{I_m \cdot I_m \cdot  E} && {\bf E}
}
$$
equals the identity map of the object ${\bf E}$. Since the abelian category $\mathrm{Mod}\text{-}\cM_\infty$ is idempotent complete, we conclude that the right $\cM_\infty$-module $\widehat{{\bf E}}$ is a direct factor of $\widehat{\bf I_m}$. This achieves the proof. 
\end{proof}
By combining the exact sequence \eqref{eq:ses1} with the derived Morita equivalence \eqref{eq:Moritaeq} we obtain an exact sequence in $\Hmo$
\begin{equation}\label{eq:ses2}
0 \too \underline{k} \too \underline{\Gamma} \too \underline{\Sigma} \too 0\,.
\end{equation}
\begin{notation}\label{not:principal}
Given a dg category $\cA$, we denote by $\Gamma(\cA)$ the dg category $\underline{\Gamma} \otimes \cA$ and by $\Sigma(\cA)$ the dg category $\underline{\Sigma} \otimes \cA$; see \S\ref{sub:bimodules}.
\end{notation}
\begin{proposition}\label{prop:main1}
For every dg category $\cA$ we have an exact sequence in $\Hmo$
$$ 0 \too \cA \too \Gamma(\cA) \too \Sigma(\cA) \too 0\,.$$
\end{proposition}
\begin{proof}
The exact sequence~\eqref{eq:ses2} and \cite[Proposition~1.6.3]{Drinfeld} lead to the exact sequence in $\Hmo$
$$ 0 \too \underline{k}\otimes^{\bbL} \cA \too \underline{\Gamma}\otimes^{\bbL}\cA \too \underline{\Sigma}\otimes^{\bbL}\cA \too 0 \,.$$
By Lemma~\ref{lem:flat} the algebras $\Gamma$ and $\Sigma$ are flat as $k$-modules and so the derived tensor products are identified in $\Hmo$ with the ordinary ones. Moreover, we have a natural isomorphism $\underline{k}\otimes^{\bbL} \cA \simeq \cA$.
\end{proof}

\section{Flasqueness of $\Gamma$}\label{sub:flasque}
\begin{definition}
Let $\cA$ be a dg category with {\em sums} (\ie the diagonal dg functor $\Delta: \cA \to \cA\times \cA$ admits a left adjoint $\oplus: \cA \times \cA \to \cA$) such that $\mathsf{Z}^0(\cA)$ is equivalent to $\perf(\cA)$; see Notations~\ref{not:basic} and \ref{not:perf}. Under these hypothesis, we say that $\cA$ is {\em flasque} if there exists a dg functor $\tau:\cA \to \cA$ and a natural isomorphism $\mathrm{Id}\oplus\tau \simeq \tau$.
\end{definition}
\begin{proposition}\label{prop:main2}
The dg category $\perf_\dg(\Gamma)$ is flasque.
\end{proposition}
\begin{proof}
Notice first that since we have an isomorphism of $k$-algebras
$$(-)^T: \Gamma^\op \stackrel{\sim}{\too} \Gamma$$
it is equivalent to show that the dg category $\perf_\dg(\Gamma^\op)$ is flasque. By definition, $\perf_\dg(\Gamma^\op)$ has sums and $\mathsf{Z}^0(\perf_\dg(\Gamma^\op))$ is equivalent to $\perf(\Gamma^\op)$. Now, recall from Definition~\ref{def:bimodule} the construction of the $\Gamma\text{-}\Gamma$-bimodule $W$. As explained in \cite[\S3.8]{ICM} the bimodule $W$ gives rise to a Quillen adjunction
$$
\xymatrix{
\cC(\Gamma^\op)\ar@<1ex>[d] \\
\cC(\Gamma^\op) \ar@<1ex>[u]^{W\otimes_\Gamma -} \,,
}
$$
which is moreover compatible with the $\cC(k)$-enrichment. Since the $\Gamma\text{-}\Gamma$-bimodule $W$ is $\Gamma$ as a left $\Gamma$-module, the left Quillen dg functor 
$$W \otimes_{\Gamma}-: \cC_\dg(\Gamma^\op) \too \cC_\dg(\Gamma^\op)$$ restricts to a dg functor

$$ \tau: \perf_\dg(\Gamma^\op) \too \perf_\dg(\Gamma^\op)\,.$$
Moreover, given an object $P$ in $\perf_\dg(\Gamma^\op)$ we have a functorial isomorphism
$$
\xymatrix{ P\oplus \tau(P) = P\oplus (W \otimes_{\Gamma}P) \simeq (\Gamma \oplus W) \otimes_{\Gamma}P \ar[r]^-{\psi}_-{\sim} & W \otimes_{\Gamma}P = \tau(P)\,,}$$
where $\psi$ is obtained by tensoring the $\Gamma\text{-}\Gamma$-bimodule isomorphism $\Gamma\oplus W \stackrel{\sim}{\to} W$ of Lemma~\ref{lem:flasque} with $P$. This achieves the proof. 
\end{proof}
\begin{lemma}\label{lem:induced}
Let $\cA$ and $\cB$ be two dg categories, with $\cA$ flasque. Then the dg category $\rep(\cB,\cA)$ (see \S\ref{sub:Morita}) is also flasque.
\end{lemma}
\begin{proof}
By construction, the dg category $\rep(\cB,\cA)$ has sums and $\mathsf{Z}^0(\rep(\cA, \cB))$ is equivalent to $\perf(\rep(\cA,\cB))$. Moreover, since $\cA$ is flasque and $\rep(\cB,-)$ is a $2$-functor which preserves (derived) products, we obtain a dg functor $$\rep(\cB,\tau):\rep(\cB,\cA) \too \rep(\cB,\cA)$$ and a natural isomorphism $\mathrm{Id}\oplus \rep(\cB, \tau)\simeq \rep(\cB,\tau)$. 
\end{proof}

Let us now recall the definition of the algebraic $K$-theory of dg categories. Given a dg category $\cA$ we denote by $\perf^{\cW}(\cA)$ the Waldhausen category $\perf(\cA)$, whose weak equivalences and cofibrations are those of the Quillen model structure on $\cC(\cA)$; see \cite[\S3]{DS}. The {\em algebraic $K$-theory spectrum $K(\cA)$ of $\cA$} is the Waldhausen's $K$-theory spectrum~\cite{Wald} of $\perf^{\cW}(\cA)$. Given a dg functor $F: \cA \to \cB$, the extension of scalars left Quillen functor $F_!: \cC(\cA) \to \cC(\cB)$ preserves weak equivalences, cofibrations,
and pushouts. Therefore, it restricts to an exact functor $F_!: \perf^{\cW}(\cA) \to \perf^{\cW}(\cB)$
between Waldhausen categories and so it gives rise to a morphism of spectra $K(F):K(\cA) \to K(\cB)$.

\begin{lemma}\label{lem:K-trivial}
Let $\cA$ be a flasque dg category. Then, its algebraic $K$-theory spectrum $K(\cA)$ is contractible.
\end{lemma}
\begin{proof}
By applying the functor $\mathsf{Z}^0(-)$ to $\cA$ and $\tau$ we obtain an exact functor $$\mathsf{Z}^0(\tau): \perf^{\cW}(\cA) \to \perf^{\cW}(\cA)$$ and a natural isomorphism $\mathrm{Id}\oplus \mathsf{Z}^0(\tau)\simeq \mathsf{Z}^0(\tau)$. Since Waldhausen's $K$-theory satisfies additivity~\cite[Proposition~1.3.2(4)]{Wald} we have the following equality
$$ K(\mathrm{Id}) + K(\mathsf{Z}^0(\tau)) = K(\mathsf{Z}^0(\tau))$$
in the homotopy category of spectra. Therefore, we conclude that $\mathrm{Id}_{K(\cA)}\simeq K(\mathrm{Id})$ is the trivial map. This shows that the algebraic $K$-theory spectrum $K(\cA)$ is contractible.
\end{proof}

\section{Proof of Theorem~\ref{thm:main}}\label{sub:proof}
We start by showing that $\Gamma(\cA)$ (see Notation~\ref{not:principal}) becomes the zero object in the triangulated category $\Mloc(e)$ after application of $\Uloc$. Since $\Gamma(\cA)\simeq \underline{\Gamma}\otimes^{\bbL}\cA$ and $\Uloc$ is symmetric monoidal with respect to a homotopy colimit preserving symmetric monoidal structure on $\Mloc$ (see \cite[Theorem~7.5]{CT1}), it suffices to show that $\underline{\Gamma}$ becomes the zero object in $\Mloc(e)$.

Recall from \cite[\S17]{Duke} that the universal localizing invariant admits the following factorization
\begin{equation}\label{eq:factorization}
\Uloc : \HO(\dgcat) \stackrel{\Uadd}{\too} \Madd \stackrel{\gamma}{\too} \Mloc\,,
\end{equation}
where $\Madd$ is the additive motivator\footnote{The additive motivator has a construction similar to the localizing one. Instead of imposing localization we impose the weaker requirement of additivity.} and $\gamma$ is a localizing morphism between triangulated derivators. Moreover, thanks to \cite[Proposition~3.7]{CT} the objects $\Uadd(\cB)[n]$, with $\cB$ a dg cell and $n \in \bbZ$, form a set of (compact) generators of the triangulated category $\Madd(e)$. Therefore, $\Uadd(\underline{\Gamma})$ is the zero object in $\Madd(e)$ if and only the spectra of morphisms $\bbR\uHom(\Uadd(\cB), \Uadd(\underline{\Gamma}))$, with $\cB$ a dg cell, is (homotopically) trivial; see \cite[\S A.3]{CT}. By \cite[Theorem~15.10]{Duke} we have the following equivalences
$$\bbR\uHom(\Uadd(\cB), \Uadd(\underline{\Gamma}))\simeq K\rep(\cB, \underline{\Gamma}) \simeq K\rep(\cB, \perf_\dg(\underline{\Gamma}))\,.$$
Therefore, Proposition~\ref{prop:main2} and Lemmas~\ref{lem:induced} and~\ref{lem:K-trivial} imply that $\Uadd(\underline{\Gamma})$ is the zero object in $\Madd(e)$. Thanks to the above factorization \eqref{eq:factorization} we conclude that $\Uloc(\underline{\Gamma})$ (and so $\Uloc(\Gamma(\cA))$) is the zero object in $\Mloc(e)$.

Now, recall from Proposition~\ref{prop:main1} that we have an exact sequence
\begin{equation*}\label{eq:ses3}
0 \too \cA \too \Gamma(\cA) \too \Sigma(\cA) \too 0\,.
\end{equation*}
By applying the universal localizing invariant to the preceding exact sequence we obtain a distinguished triangle
$$ \Uloc(\cA) \too \Uloc(\Gamma(\cA)) \too \Uloc(\Sigma(\cA)) \too \Uloc(\cA)[1]$$
in $\Mloc(e)$. Therefore, since $\Uloc(\Gamma(\cA))$ is the zero object, we have a canonical isomorphism
$$ \Uloc(\Sigma(\cA)) \stackrel{\sim}{\too} \Uloc(\cA)[1]\,.$$


\begin{thebibliography}{00}

\bibitem{BM} A.~Blumberg and M.~Mandell, {\em Localization theorems in topological
Hochschild homology and topological cyclic homology}. Available at arXiv:$0802.3938$.    
    
\bibitem{BKap} A.~Bondal and M.~Kapranov, {\em Framed triangulated
    categories} (Russian) Mat.~Sb. {\bf 181} (1990) no.~5, 669--683;
    translation in Math.~USSR-Sb. {\bf 70} no.~1, 93--107.

\bibitem{Bvan} A.~Bondal and M.~Van den Bergh, {\em Generators and Representability
of Functors in commutative and Noncommutative Geometry}.
Moscow Math. Journal {\bf 3} (2003), no.~1, 1--37.

\bibitem{Cortinas} G.~Corti{\~n}as and A.~Thom, {\em Bivariant algebraic $K$-theory}. J.~Reine Angew. Math {\bf 510}, 71--124.

\bibitem{CT} D.-C.~Cisinski and G.~Tabuada, {\em Non-connective $K$-theory via universal invariants}.
Available at arXiv:$0903.3717\textrm{v}2$.

\bibitem{CT1} \bysame, {\em Symmetric monoidal structure on non-commutative motives}.
Available at arXiv:$1001.0228\textrm{v}2$. 

\bibitem{Drinfeld} V.~Drinfeld, {\em DG quotients of DG categories}.
J. Algebra {\bf 272} (2004), 643--691.

\bibitem{DS} D.~Dugger and B.~Shipley, {\em $K$-theory and derived equivalences}, Duke Math. J. {\bf 124} (2004), no.~3, 587--617.


\bibitem{Grothendieck} A.~Grothendieck, {\em Les D{\'e}rivateurs}.
Available at \url{http://people.math.jussieu.fr/maltsin/groth/Derivateurs.html}.

\bibitem{Kaledin} D.~Kaledin, {\em Motivic structures in non-commutative geometry}. Available at arXiv:$1003.3210$. To appear in the Proceedings of the ICM 2010.

\bibitem{Kassel} C.~Kassel, {\em Caract{\`e}re de Chern bivariant}. $K$-theory {\bf 3} (1989), 367-400.

\bibitem{Karoubi} M.~Karoubi and O.~Villamayor, {\em $K$-th{\'e}orie algebrique et $K$-th{\'e}orie topologique}. Math. Scand. {\bf 28} (1971), 265--307.

\bibitem{ICM} B.~Keller, {\em On differential graded
    categories}. ICM (Madrid), Vol.~II,
  151--190. Eur.~Math.~Soc., Z{\"u}rich, 2006.

\bibitem{Exact} \bysame, {\em On the cyclic homology of exact categories}. J.~Pure~Appl.~Algebra {\bf 136} (1999), no.~1,~1--56.
    
\bibitem{Exact1} \bysame, {\em On the cyclic homology of ringed spaces and schemes}. Doc. Math. {\bf 3} (1998), 231--259.

\bibitem{Keller} \bysame, {\em Invariance and localization for cyclic homology of DG algebras}. J.~Pure~Appl.~Algebra {\bf 123} (1998), 223--273.

\bibitem{IAS} M.~Kontsevich, {\em Non-commutative motives}.
Talk at the IAS on the occasion of the 61st birthday of Pierre Deligne, October 2005.
Video available at \url{http://video.ias.edu/Geometry-and-Arithmetic}.

\bibitem{Kontsevich-Langlands} \bysame, {\em Notes on motives in finite characteristic}.
Available at arXiv:0702206. To appear in ``Manin's Festschrift''.

\bibitem{Lam} T.~Y.~Lam, {\em Lectures on modules and rings}. Graduate Texts in Mathematics {\bf 189}, Springer.

\bibitem{Neeman} A.~Neeman, {\em Triangulated categories}. Annals of Mathematics Studies {\bf 148}. Princeton University Press, 2001.

\bibitem{Orlov} V.~Lunts and D.~Orlov, {\em Uniqueness of enhancement for triangulated categories}. Available at arXiv:$0908.4187\textrm{v}5$.

\bibitem{Marco} M.~Schlichting, {\em Negative $\mbox{K}$-theory of
    derived categories}. Math.~Z. {\bf 253} (2006), no.~1, 97--134.

\bibitem{Schlichting} \bysame, {\em The Mayer-Vietoris principle for Grothendieck-Witt groups of schemes}. Invent. Math. {\bf 179} (2010), no.~2, 349--433.

\bibitem{Schwede} S.~Schwede, {\em Morita theory in abelian, derived and stable model categories}. Structured Ring Spectra, London Math. Society Lecture Notes {\bf 315}, 33--86. Cambridge Univ. Press, 2004.

\bibitem{Thesis} G.~Tabuada, {\em Th{\'e}orie homotopique des DG-cat{\'e}gories}. Author's Ph.D thesis. Available at arXiv:$0710.4303$. 

\bibitem{Duke} \bysame, {\em Higher $K$-theory via universal invariants}.
Duke Math. J. {\bf 145} (2008), no.~1, 121--206.

\bibitem{IMRN} \bysame, {\em Invariants additifs de dg-cat{\'e}gories}. Int.~Math.~Res.~Not. {\bf 53} (2005), 3309--3339.  

\bibitem{IMRNC} \bysame, {\em Corrections to ``Invariants additifs de dg-cat{\'e}gories''.} Int.~Math.~Res.~Not. {\bf 24} (2007), art. ID rnm149.

\bibitem{AGT} \bysame, {\em Generalized spectral categories, topological
Hochschild homology, and trace maps}. Algebraic and Geometric Topology, {\bf 10} (2010), 137--213.

\bibitem{Thomason} R.W.~Thomason and T.~Trobaugh,
{\em Higher algebraic K -theory of schemes and of derived 
categories}. Grothendieck Festschrift, Volume III. Volume {\bf 88} of
Progress in Math., 247--436. Birkhauser, Boston, Bassel, Berlin, 1990. 

\bibitem{Toen} B.~To{\"e}n, {\em The homotopy theory of dg-categories and
    derived Morita theory}. Invent. Math. {\bf 167} (2007), no.~3, 615--667.

\bibitem{Wald} F.~Waldhausen, {\em Algebraic K-theory of spaces}.
  Algebraic and geometric topology (New Brunswick, N.~J., 1983),
  318--419. Lecture Notes in Math. {\bf 1126}, Springer, Berlin, 1985.

\bibitem{Weibel} C.~Weibel, {\em Introduction to homological algebra}. Cambridge Studies in Adv.~Mathematics {\bf 38}.

\end{thebibliography}
\end{document}